\RequirePackage{fix-cm}

\documentclass[12pt, reqno]{amsart}

\usepackage{amssymb, amsmath, amsthm, multirow}
\usepackage[backref]{hyperref}
\usepackage[alphabetic,backrefs,lite]{amsrefs}
\usepackage{amscd}   
\usepackage{fullpage}
\usepackage[all]{xy} 
\usepackage{lscape}

\DeclareFontEncoding{OT2}{}{} 

\usepackage{color}

\newtheorem{lemma}{Lemma}[section]

\newtheorem{proposition}[lemma]{Proposition}
\newtheorem{prop}[lemma]{Proposition}
\newtheorem{cor}[lemma]{Corollary}

\newtheorem{claim*}{Claim}
\newtheorem{thm}[lemma]{Theorem}

\theoremstyle{definition}
\newtheorem{remark}[lemma]{Remark}
\newtheorem{remarks}[lemma]{Remarks}

\newcommand{\Aff}{{\mathbb A}}
\newcommand{\G}{{\mathbb G}}

\newcommand{\PP}{{\mathbb P}}

\newcommand{\F}{{\mathbb F}}
\newcommand{\Q}{{\mathbb Q}}
\newcommand{\R}{{\mathbb R}}
\newcommand{\Z}{{\mathbb Z}}

\newcommand{\Xbar}{{\overline{X}}}
\newcommand{\Qbar}{{\overline{\Q}}}

\newcommand{\kbar}{{\overline{k}}}

\newcommand{\Ybar}{{\overline{Y}}}

\newcommand{\Adeles}{{\mathbb A}}
\newcommand{\kk}{{\mathbf k}}

\newcommand{\Qtilde}{\widetilde{Q}}
\newcommand{\Ctilde}{\widetilde{C}}
\newcommand{\Dtilde}{\widetilde{D}}

\newcommand{\Kabc}{K_{a,b,c}}
\newcommand{\Xabc}{X_{\mathbf{a}}}
\newcommand{\Yabc}{Y_{\mathbf{a}}}
\newcommand{\Yabcp}{Y_{\mathbf{a},p}}
\newcommand{\Yabcseven}{Y_{\mathbf{a},7}}
\newcommand{\Yabceleven}{Y_{\mathbf{a},11}}
\newcommand{\Yabcpi}{Y_{\mathbf{a},p_i}}
\newcommand{\Xabcbar}{\overline{X}_{\mathbf{a}}}
\newcommand{\Yabcbar}{\overline{Y}_{\mathbf{a}}}
\newcommand{\Yabcpbar}{\overline{Y}_{\mathbf{a},p}}
\newcommand{\Yabcsevenbar}{\overline{Y}_{\mathbf{a},7}}
\newcommand{\Yabcelevenbar}{\overline{Y}_{\mathbf{a},11}}

\newcommand{\calA}{{\mathcal A}}

\newcommand{\OO}{{\mathcal O}}

\DeclareMathOperator{\HH}{H}

\DeclareMathOperator{\rk}{rk}

\DeclareMathOperator{\im}{im}

\DeclareMathOperator{\Gal}{Gal}

\DeclareMathOperator{\Br}{Br}

\DeclareMathOperator{\Pic}{Pic}

\DeclareMathOperator{\Num}{Num}

\DeclareMathOperator{\Proj}{Proj}

\DeclareMathOperator{\et}{et}

\DeclareMathOperator{\red}{red}

\newcommand{\isom}{\cong}

\newcommand{\hideqed}{\renewcommand{\qed}{}} 

\numberwithin{equation}{section}
\numberwithin{table}{section}

\newcommand{\defi}[1]{\textsf{#1}} 

\title[Failure of the Hasse principle]{Failure of the Hasse principle for Enriques surfaces}

\author{Anthony V\'arilly-Alvarado}
\author{Bianca Viray}
\thanks{The second author was partially supported by NSF Grant DMS-0841321 and a Ford Foundation Dissertation Fellowship.  This collaboration was partially supported by Rice University.}

\address{Department of Mathematics, Rice University, Houston, TX 77005, USA}
\email{varilly@rice.edu}
\urladdr{http://www.math.rice.edu/\~{}av15}
\address{Department of Mathematics, Box 1917, Brown University, Providence, RI
			02912, USA}
\email{bviray@math.brown.edu}
\urladdr{http://math.brown.edu/\~{}bviray}

\date{}

\begin{document}
	
	\fontsize{11.8}{13.5}\selectfont
	\vspace{-.35in}
	\begin{abstract}
		We construct an Enriques surface $X$ over $\Q$ with empty \'etale-Brauer set (and hence no rational points) for which there is no algebraic Brauer-Manin obstruction to the Hasse principle. In addition, if there is a transcendental obstruction on $X$, then we obtain a K3 surface that has a transcendental obstruction to the Hasse principle.
	\end{abstract}
	
	
	\maketitle
	\vspace{-.1in}
	\section{Introduction}

		Let $X$ be a smooth, projective, geometrically integral scheme over a 
		number field $k$.  We say that $X$ satisfies the \defi{Hasse principle}
		if the set $X(k)$ of $k$-rational points is nonempty whenever the set of 
		adelic points $X(\Adeles_k)$ is also nonempty.  Manin and Skorobogatov 
		have defined intermediate ``obstruction sets'' that fit between $X(k)$ 
		and $X(\Adeles_k)$ (cf. \S\ref{sect: background} 
		or~\cite{Manin-obstruction,Skorobogatov-beyondmanin}):
			\begin{equation}
				\label{eq: obstruction sets}
				X(k) \subseteq X(\Adeles_k)^{\et,\Br} \subseteq 
				X(\Adeles_k)^{\Br} \subseteq
				X(\Adeles_k)^{\Br_1} \subseteq X(\Adeles_k).
			\end{equation}
		Lind and Reichardt, Harari, Skorobogatov, and Poonen constructed the 
		first schemes that show, respectively, that each above containment (from 
		right to left) can be 
		strict~\cite{Lind-HP,Reichardt-HP,Harari-transcendental, 
		Skorobogatov-beyondmanin, Poonen-insufficiency}. 

		A wide open area of research considers the finer question: to what 
		extent do the sets in~\eqref{eq: obstruction sets} give distinct 
		obstructions to the Hasse principle \emph{after} fixing some numerical 
		invariants, like dimension, of X? For curves, Scharaschkin and 
		Skorobogatov independently asked if $X(\Adeles_k)^{\Br_1} \neq 
		\emptyset$ implies that $X(k) \neq \emptyset$, i.e., if the algebraic 
		Brauer-Manin obstruction explains all counterexamples to the Hasse 
		principle~\cite{Scharaschkin-thesis, Skorobogatov-torsors}.  (This 
		question has since been upgraded to a conjecture~\cite{Flynn-BM, 
		Poonen-heuristics, Stoll-BM}.)  In the case of geometrically rational 
		surfaces, Colliot-Th\'el\`ene and Sansuc conjectured that the same 
		implication holds~\cite{CTS-descent}. In contrast, for most other 
		Enriques-Kodaira classes of surfaces we expect that this is no longer 
		the case.

		However, there are strikingly few examples of surfaces that corroborate 
		this expectation.  In a pioneering paper, Skorobogatov constructs the 
		first surface for which the failure of the Hasse principle is not 
		explained by an algebraic Brauer-Manin obstruction (or, for that matter, 
		a transcendental Brauer-Manin 
		obstruction)~\cite{Skorobogatov-beyondmanin}; the other known example, 
		due to Basile and Skorobogatov, is of a similar 
		nature~\cite{BS-beyondmanin}. In both cases, the surfaces considered are 
		bi-elliptic and the failure is caused by an \'etale-Brauer obstruction, 
		i.e. $X(\Adeles_k)^{\et, \Br} = \emptyset$.

		We show that Enriques surfaces give rise to a similar insufficiency 
		phenomenon. More precisely, our main result is as follows.
		\begin{thm}\label{thm:main}
			There exists an Enriques surface $X/\Q$ such that 
			\[
				X(\Adeles_{\Q})^{\et,\Br} = \emptyset \quad\textup{and}\quad
				X(\Adeles_{\Q})^{\Br_1} \neq \emptyset.
			\]
			Moreover, if $X(\Adeles_{\Q})^{\Br} = \emptyset$, then 
			$Y(\Adeles_{\Q})^{\Br Y \setminus\Br_1Y} = \emptyset$, where $Y$ is 
			a K3 double cover of $X$.
		\end{thm}
		Note the rather curious dichotomy: we obtain either a K3 surface with a 
		transcendental Brauer-Manin obstruction to the Hasse principle, or an 
		Enriques surface whose failure of the Hasse principle is unaccounted for 
		by a Brauer-Manin obstruction.

		It is important to remark that Cunnane already showed that Enriques 
		surfaces need not satisfy the Hasse principle~\cite{Cunnane-thesis}; 
		however, his counterexamples are explained by an algebraic Brauer-Manin 
		obstruction.  On the other hand, Harari and Skorobogatov, and later 
		Cunnane, showed that the Brauer-Manin obstruction is insufficient to 
		explain all failures of weak approximation on Enriques 
		surfaces~\cites{HS-Enriques, Cunnane-thesis}, thereby opening up the 
		analogous question for the Hasse principle.  Theorem~\ref{thm:main} 
		represents a key step towards a complete answer to this question.

		\subsection{Outline of proof}
		\label{ss: outline}

			The proof of Theorem~\ref{thm:main} is constructive. Let 
			$\mathbf{a} := (a,b,c) \in\Z^3$, and consider the intersection 
			$\Yabc$ of the three quadrics
				\begin{align*}
					xy+ 5z^2 & = s^2\\
					(x + y)(x + 2y) & = s^2 - 5t^2\\
					ax^2 + by^2 + cz^2 & = u^2.
				\end{align*}
			 in $\PP^5 = \Proj\Q[s,t,u,x,y,w]$. Suppose that
				\[
					abc(5a + 5b + c)(20a + 5b + 2c)(4a^2 + b^2)(c^2 - 100ab)
					(c^2 + 5bc + 10ac + 25ab) \neq 0.
				\] 
			Then $\Yabc$ is smooth and thus defines a K3 surface.  The 
			involution 
				\[
					\sigma:\PP^5 \to \PP^5,\qquad 
					(s:t:u:x:y:z)\mapsto(-s:-t:-u:x:y:z)
				\]
			has no fixed points when restricted to $\Yabc$, so 
			$\Xabc := \Yabc/\sigma$ is an Enriques surface.  

			\begin{thm}\label{thm:main-abc}
				Let $\mathbf{a} = (a,b,c) \in \Z_{>0}^3$ satisfy the following 
				conditions:
				\begin{enumerate}
					\item for all prime numbers $p\ |\ (5a + 5b + c)$, 
							$5$ is not a square modulo $p$,
					\item for all prime numbers $p\ |\ (20a + 5b + 2c)$,  
							$10$ is not a square modulo $p$,
					\item the quadratic form $ax^2 + by^2 + cz^2 + u^2$ is 
							anisotropic over $\Q_3$,
					\item the integer $-bc$ is not a square modulo $5$,
					\item the triplet $(a,b,c)$ is congruent to $(5,6,6)$ 
							modulo $7$,
					\item the triplet $(a,b,c)$ is congruent to $(1,1,2)$ 
							modulo $11$,
					\item $\Yabc(\Adeles_{\Q}) \neq \emptyset$, and 
					\item the triplet $(a,b,c)$ is \defi{Galois general}, 
					meaning that a certain number field defined in terms of 
					$a,b,c$ is as large as possible.  A precise definition is 
					given in \S\ref{subsec:galoisgeneral}.
				\end{enumerate}
				Then 
					\[
						\Xabc(\Adeles_{\Q})^{\et,\Br} = \emptyset 
						\quad\textup{and}\quad
						\Xabc(\Adeles_{\Q})^{\Br_1} \neq \emptyset.
					\]
				Moreover, if $\Xabc(\Adeles_{\Q})^{\Br} = \emptyset$, then  
				$\Yabc(\Adeles_{\Q})^{\Br\setminus\Br_1} = \emptyset$.
			\end{thm}

			Theorem~\ref{thm:main} follows almost at once from 
			Theorem~\ref{thm:main-abc}: we show that the triplet $\mathbf{a} = 
			(12,111,13)$ satisfies conditions $(1)$--$(8)$.

			\begin{remarks}
				\begin{enumerate}
					\item[(i)] In lieu of a paper outline, let us explain the 
					role that the assumptions of Theorem~\ref{thm:main-abc} play 
					in our constructions. In \S\ref{sec:noQpoints}, we show that 
					conditions $(1)$--$(4)$ imply the \'etale-Brauer set of 
					$\Xabc$ is empty.  In \S\ref{sec:Picard}, we use conditions 
					$(5)$, $(6)$, and $(8)$ to describe explicit generators for 
					the Picard groups of $\Xabc$ and $\Yabc$, as well as to 
					compute the low degree Galois cohomology of these groups.  
					In \S\ref{sec:BrSet}, we use the results from 
					\S\ref{sec:Picard} together with conditions $(5)$--$(8)$ to 
					determine both the Brauer set and algebraic Brauer set of 
					$\Xabc$.  Finally, in \S\ref{sec:proofofprops}, we prove 
					Theorems~\ref{thm:main-abc} and~\ref{thm:main}.

					\item[(ii)] Our construction relies heavily on an example 
					due to Birch and Swinnerton-Dyer of a del Pezzo surface of 
					degree $4$ that violates the Hasse 
					principle~\cite{BSD-hasse}.  While we expect the 
					construction will work with other such surfaces, there are a 
					few subtleties which are not readily apparent in the 
					argument.  We elaborate on this point in 
					\S\ref{subsec:dp4details}.
				\end{enumerate}
			\end{remarks}

		\subsection{Notation}
			Throughout $k$ denotes a perfect field, $\kbar$ is a fixed algebraic 
			closure and $G_k$ denotes the absolute Galois group $\Gal(\kbar/k)$.  
			For any $k$-scheme $X$, we write $\Xbar$ for the base change 
			$X\times_k\kbar$ and, if $X$ is integral, $\kk(X)$ for the function 
			field of $X$.  For a smooth, projective, geometrically integral 
			variety $X$ we identify $\Pic X$ with the Weil class group; in 
			particular, we use additive notation for the group law on $\Pic X$. 
			In addition, we write $K_X$ for the class of the canonical sheaf in 
			the group $\Pic X$. Finally, we denote by $\rho(X)$ the geometric 
			Picard number of $X$, i.e., the rank of the N\'eron-Severi group of 
			$\Xbar$.

			For a homogeneous ideal $I$ in a graded ring $R$, we write $V(I)$ 
			for $\Proj R/I$; if explicit generators of $I = \langle f_i : i \in 
			S\rangle$ are given, then we write $V(f_i : i \in S)$ instead of 
			$V(I)$.

			Henceforth, ``condition(s)'' refers to items $(1)$--$(8)$ in 
			Theorem~\ref{thm:main-abc}.

		\subsection*{Acknowledgements}
			We thank Bjorn Poonen and Brendan Hassett for many helpful 
			conversations.  We are indebted to Michael Stoll and Damiano Testa 
			for providing the ideas behind the proof of 
			Prop~\ref{prop:saturation}.  We also thank Jean-Louis 
			Colliot-Th\'el\`ene, Daniel Erman, Matthias Sch\"utt, Damiano Testa, 
			Yuri Tschinkel, and Olivier Wittenberg for several comments.  All 
			computations were done using {\tt Magma}~\cite{magma}.	

	\section{Background}\label{sect: background}

		\subsection{K3 surfaces and Enriques surfaces}
			Assume that the characteristic of $k$ is not $2$.  A \defi{K3 
			surface} is a smooth projective $k$-surface $X$ of Kodaira dimension 
			$0$ with trivial canonical divisor and $h^1(X,\OO_X) = 0$.  The 
			geometric Picard group of a K3 surface is a free abelian group of 
			rank at most $22$ (in characteristic $0$, the rank is at most $20$).  
			In addition, the intersection lattice of a K3 surface can be 
			embedded primitively in $U^{\oplus3}\oplus E_8(-1)^{\oplus2}$, the 
			unique even unimodular lattice of signature $(3,19)$.

			An \defi{Enriques surface} is a smooth projective $k$-surface $X$ of 
			Kodaira dimension $0$ with numerically trivial canonical divisor and 
			second Betti number $b_2 = 10$.  These assumptions imply that $K_X 
			\neq 0$ and $2K_X = 0$.  Enriques surfaces are also characterized in 
			terms of K3 surfaces as 
			follows~\cite[Proposition~1.3.1]{CD-Enriques}
			~\cite[Proposition~VIII.17]{Beauville-surfaces}.
			\begin{thm}
				Let $X$ be an Enriques surface and let $f\colon Y\to X$ be any 
				\'etale double cover associated to 
				$K_X \in \left(\Pic X\right)[2]$.  Then $Y$ is a K3 surface.  
				Conversely, the quotient of a K3 surface by a fixed-point free 
				involution is an Enriques surface.
			\end{thm}

			\subsubsection{Geometric birational models of Enriques surfaces}
			\label{subsec:ModelsEnr}
				Assume that $k$ is algebraically closed.  A generic Enriques 
				surface can be constructed as follows.  Consider a K3 surface 
				$Y$ of degree $8$ in $\PP^5_k = \Proj k[s,t,u,x,y,z]$ given by
					\begin{equation}
						\label{eq: general double cover}
						V\big( Q_i(s,t,u) - \Qtilde_i(x,y,z):
						\quad i = 1, 2, 3 \big),
					\end{equation}
				where $Q_i \in k[s,t,u]$ and $\Qtilde_i \in k[x,y,z]$ are 
				quadratic polynomials for $i = 1,2,3$. For generic $Q_i$ and 
				$\Qtilde_i$, the involution $\sigma$ of \S\ref{ss: outline} has 
				no fixed points when restricted to $Y$, so $X := Y/\sigma$ is an 
				Enriques surface.
				
				Cossec and Verra show that a stronger statement holds:  Let 
				$f\colon Y \to X$ be the double cover of an Enriques surface $X$ 
				by a $K3$ surface $Y$. Then there exists a birational map 
				between $Y$ and a surface of the form~\eqref{eq: general double 
				cover} that identifies $\sigma$ with the unique fixed-point free 
				involution $\iota$ such that 
				$f\circ\iota=f$~\cite{Cossec-projmodels, Verra-Enriques}.

			\subsubsection{The Picard group}\label{subsec:PicEnr}
				For any Enriques surface $X$, we have the following exact 
				sequence of Galois modules~\cite[Theorem~1.2.1 and 
				Proposition~1.2.1]{CD-Enriques}
				\begin{equation}
					\label{eqn:PicEnr}
					0\to \langle K_X\rangle \to \Pic \Xbar \to \Num \Xbar \to 0. 
				\end{equation}
				Additionally, $\Num \Xbar \isom U \oplus E_8(-1)$, the unique 
				even unimodular lattice of rank $10$ and signature 
				$(1,9)$~\cite[Theorem~2.5.1]{CD-Enriques}.  We also have an 
				exact sequence relating the geometric Picard group of $X$ and 
				that of any K3 double cover $Y$ of 
				$X$~\cite[Equation~(3.15)]{HS-Enriques}
				\begin{equation}
					\label{eq: Pic fixed by sigma}
					0 \to \Z/2\Z \to \Pic \Xbar \to 
						\left(\Pic \Ybar\right)^{\sigma} \to 0.
				\end{equation}
				Combining~\eqref{eqn:PicEnr} and~\eqref{eq: Pic fixed by sigma} 
				we obtain an isomorphism of Galois modules $\Num\Xbar\cong 
				\left(\Pic\Ybar\right)^\sigma$.
		
		\subsection{The Brauer group and the Brauer-Manin obstruction}

			For the remainder of \S\ref{sect: background}, we restrict to the 
			case where $k$ is a number field.  Let $\Br X := 
			\HH^2_{\et}(X,\G_m)$ be the Brauer group of a variety $X$.  An 
			element $\calA \in \Br X$ is \defi{algebraic} if it belongs to the 
			subgroup $\Br_1 X := \ker\left(\Br X \to \Br \Xbar\right)$; 
			otherwise $\calA$ is called \defi{transcendental}.

			Functoriality of the Brauer group yields an evaluation pairing
			\[
				\langle\,\cdot\, , \cdot\,\rangle\colon 
				X(\Adeles_k) \times \Br X {\longrightarrow} \Q/\Z.
			\]
			For any set $S \subseteq \Br X$ ($S$ need not be a subgroup), this 
			pairing is used to define the set
			\[
				X(\Adeles_k)^{S} = \left\{ (P_v)_v \in X(\Adeles_k) : 
				\langle (P_v), \calA\rangle = 0 
				\textup{ for all } \calA \in S \right\}.
			\]
			Class field theory guarantees that $X(k) \subseteq 
			X(\Adeles_k)^{S}$, for any $S$. When $S = \Br X$ and $\Br_1 X$, 
			respectively, we obtain the \defi{Brauer set} of $X$ and the 
			\defi{algebraic Brauer set} of $X$; we denote these sets by 
			$X(\Adeles_k)^{\Br}$ and $X(\Adeles_k)^{\Br_1}$. In summary, we have
			\[
				X(k) \subseteq X(\Adeles_k)^{\Br} \subseteq X(\Adeles_k)^{\Br_1}
				\subseteq X(\Adeles_k).
			\]
			We say there is a \defi{Brauer-Manin obstruction} to the Hasse 
			principle if $X(\Adeles_k)^{\Br} = \emptyset$ and $X(\Adeles_k) \neq 
			\emptyset$.  The obstruction is \defi{algebraic} if in addition 
			$X(\Adeles_k)^{\Br_1} = \emptyset$. See	
			~\cite[\S5.2]{Skorobogatov-torsors} for more details.

		\subsection{Torsors under finite \'etale groups and the \'etale-Brauer %
		obstruction}

			Let $G$ be an fppf group scheme over a scheme $X$.  Recall that a 
			\defi{(right) $G$-torsor over $X$} is an fppf $X$-scheme $Y$ 
			equipped with a right $G$-action such that the morphism $Y\times_X G 
			\to Y\times_X Y$ given by $(y, g) \mapsto (y, yg)$ is an 
			isomorphism. A detailed account of torsors is given in~\cite[Part 
			I]{Skorobogatov-torsors}. A torsor $f\colon Y\to X$ under a finite 
			\'etale $k$-group scheme $G$ determines a partition
			\[
				X(k) = \bigcup_{\tau \in \HH^1(k, G)} 
					f^{\tau}\left(Y^{\tau}(k)\right),
			\]
			where $Y^{\tau}$ is the twisted torsor associated to $\tau$ 
			(see~\cite[Lemma 2.2.3]{Skorobogatov-torsors}). Running over all 
			possible $G$-torsors of this form, we assemble the 
			\defi{\'etale-Brauer set}
			\[
				X(\Adeles_k)^{\et, \Br} := \bigcap_{\substack{f\colon Y\to X\\ 
				\textup{ torsor under }\\ \textup{ finite \'etale }G}} 
				\left(\bigcup_{\tau\in \HH^1(k,G)} 
				f^{\tau}\left(Y^{\tau}(\Adeles_k)^{\Br}\right)\right).
			\]
			By construction, we have $X(k) \subseteq X(\Adeles_k)^{\et, \Br}$; 
			we say there is an \defi{\'etale-Brauer obstruction} to the Hasse 
			principle if $X(\Adeles_k) \neq \emptyset$ and $X(\Adeles_k)^{\et, 
			\Br} = \emptyset$.  Note that $X(\Adeles_k)^{\et, \Br} \subseteq 
			X(\Adeles_k)^{\Br}$, so the \'etale-Brauer obstruction is at least 
			as strong as the Brauer-Manin obstruction.

			\subsubsection{K3 double covers as torsors}\label{ss:torsors}

				Any K3 double cover $f\colon Y \to X$ of an Enriques surface $X$ 
				is a $\Z/2\Z$-torsor over $X$. By Kummer theory, we have 
				$\HH^1(k,\Z/2\Z) = k^\times/(k^\times)^2$, so a class $\tau$ may 
				be represented by an element $d \in k^\times$, up to squares.  
				If $Y$ has the form~\eqref{eq: general double cover}, then the 
				twisted torsor $Y^\tau$ is given explicitly as
				\begin{equation}
					\label{eq: twisted double cover}
					V\big(dQ_i(s,t,u) - \Qtilde_i(x,y,z):\quad i = 1,2,3 \big).
				\end{equation}
				Over $\Q$, a class $\tau$ is represented uniquely by a 
				squarefree integer $d$; we write $\Yabc^{(d)}$ instead of 
				$\Yabc^{\tau}$.

	\section{Absence of $\Q$-points}\label{sec:noQpoints}

		\begin{lemma}\label{lem:no-Qp-points}
			Let $\mathbf{a}\in \Z^3_{>0}$ satisfy conditions $(1)$ and $(2)$. If 
			$d$ is a squarefree integer divisible by a prime $p$ different from 
			$2$ and $5$, then $\Yabc^{(d)}(\Z/p^2\Z) = \emptyset$.
		\end{lemma}
		\begin{proof}
			Let $(s:t:u:x:y:z)$ be a primitive $\Z/p^2\Z$-point of 
			$\Yabc^{(d)}$.  The equations of $\Yabc^{(d)}$ imply that
			\begin{equation}
				\label{eq: LHS eq 0 mod p}
				xy + 5z^2,\quad
				(x + y)(x + 2y), \quad\textup{and}\quad
				ax^2 + by^2 + cz^2
			\end{equation}
			are all congruent to $0$ modulo $p$.  Let us first assume that $x 
			\equiv -y \pmod p$.  Substituting this congruence into the first and 
			third quadrics of~\eqref{eq: LHS eq 0 mod p} we obtain
			\begin{align*}
				y^2 &\equiv 5z^2 \pmod p,\\
				(a + b)y^2 &\equiv -cz^2 \pmod p.
			\end{align*}
			These congruences are either linearly independent, in which case 
			$z\equiv y \equiv x \equiv 0 \pmod p$, or else $p\ |\ (5a + 5b +c)$.  
			In the latter case, by condition $(1)$ we know that $5$ is not a 
			square modulo any prime $p$ dividing $5a + 5b + c$, so the 
			congruences again imply that $z\equiv y \equiv x \equiv 0 \pmod p$.  
			The defining equations of $\Yabc^{(d)}$ then imply that $s\equiv t 
			\equiv u \equiv 0 \pmod p$, a contradiction.  The argument for the 
			case $x \equiv -2y \pmod p$ is similar.
		\end{proof}

		\begin{proposition}\label{prop:noQpoints}
			Let $\mathbf{a} \in \Z^3_{>0}$ satisfy conditions $(1)$--$(4)$.  
			Then 
			\[
				\Xabc(\Adeles_{\Q})^{\et, \Br} = \emptyset.
			\]
		\end{proposition}
		\begin{proof}
			By definition of the \'etale-Brauer set, it suffices to show that
			\begin{equation}
				\label{eq: no local pts on twists}
				\Yabc^{(d)}(\Adeles_{\Q})^{\Br} = \emptyset
			\end{equation}
			for all squarefree integers $d$ (including $1$).  
			Lemma~\ref{lem:no-Qp-points} establishes~\eqref{eq: no local pts on 
			twists} for all such $d$ except those in $\langle -1, 2, 5 \rangle$.  
			Since $a,b$ and $c$ are positive, it is easy to see that 
			$\Yabc^{(d)}(\R) =\emptyset$ for any negative $d$.  Condition $(3)$ 
			implies that $\Yabc^{(d)}(\Q_3) = \emptyset$ for all $d \equiv 
			2\pmod 3$ and condition $(4)$ implies that $Y^{(10)}(\Q_5) = 
			\emptyset$.  Thus, the only remaining case is $d = 1$.  Birch and 
			Swinnerton-Dyer prove that the del Pezzo surface $S \subset \PP^4$ 
			of degree $4$ given by 
			\begin{align}
				\label{eq:dp4-1}
				xy &= s^2 - 5z^2, \\
				(x + y)(x + 2y) &= s^2 - 5t^2,\label{eq:dp4-2}
			\end{align}
			has no adelic points orthogonal to the quaternion Azumaya algebra 
			\[
				\calA := \left(5, \frac{x+y}{x}\right) \in 
				\im\left(\Br S \to \Br\kk(S)\right);
			\]
			see~\cite{BSD-hasse}.  The existence of a map $h\colon\Yabc \to S$ 
			and functoriality of the evaluation pairing imply that 
			$\Yabc(\Adeles_{\Q})^{h^*(\calA)} = \emptyset$.
		\end{proof}

	\section{Picard Groups}\label{sec:Picard}

		We compute explicit presentations for the groups $\Pic\Xabcbar$, 
		$\Num\Xabcbar$, and $\Pic\Yabcbar$, and thus compute their Galois 
		cohomology.  In \S\ref{subsec:fibrations}, we describe how to obtain 
		genus $1$ fibrations on $\Yabcbar$ and $\Xabcbar$.  In 
		\S\ref{subsec:rank}, we prove that the fibers of these fibrations give 
		us a finite index subgroup of $\Pic\Yabcbar$, and in 
		\S\ref{subsec:saturation}, we explain how to compute its saturation in 
		$\Pic\Yabcbar$.  In \S\ref{subsec:Enriques}, we give explicit generators 
		for $\Num\Xabcbar$ in terms of the fibrations. Finally, in 
		\S\ref{subsec:cohom}, we use the explicit computations from the previous 
		sections to determine the low degree Galois cohomology of $\Pic 
		\Yabcbar, \Pic \Xabcbar,$ and $\Num \Xabcbar$.

		\subsection{Genus $1$ fibrations}\label{subsec:fibrations}
			Throughout this subsection, we work over an algebraically closed 
			field. 

			Let $Y$ be a $K3$ surface of degree $8$, given as a closed subscheme 
			of $\PP^5$ by the vanishing of a net of quadrics.  Let $Z$ be the 
			sextic curve in $\PP^2$ that parametrizes the degeneracy locus of 
			this net.  Expanding on~\cite[Example IX.4.5]{Beauville-surfaces}, 
			we explain how an isolated singular point $P \in Z$ gives rise to 
			two distinct genus $1$ fibrations on $Y$.

			Let $Q$ be the quadric corresponding to $P$ in the net defining $Y$.  
			Then $Q$ has rank $4$, and there are two rulings on it, each 
			realizing $V(Q)$ as a $\PP^3$-bundle over $\PP^1$. Restricting to 
			$Y$, we obtain two maps $\phi_P, \phi'_P\colon Y \to \PP^1_\Qbar$ 
			whose respective general fibers are smooth complete intersections of 
			two quadrics in $\PP^3$, i.e., smooth genus $1$ curves. We write 
			$F_P$ (resp.\ $G_P$) for the class in $\Pic Y$ of a fiber in 
			$\phi_P$ (resp.\ $\phi'_P$).  

			For the family of K3 surfaces given by $\Yabcbar$, the curve $Z 
			\subset \PP^2$ is the union of $4$ lines, each defined over $\Q$, 
			and a conic.  If $2a \neq b$, then the conic is geometrically 
			irreducible and there are exactly $14$ distinct singular points 
			$P_1, \ldots, P_{14}$ on $Z$. These points give rise to the $28$ 
			classes $F_i := F_{P_i}, G_i := G_{P_i}, i = 1, \ldots, 14$ in 
			$\Pic\Yabcbar$, as above. The class $F_i + G_i$ is equivalent to a 
			hyperplane section for all $i$, and we have
			\[
				F_i^2 = G_i^2 = 0, \quad F_i\cdot G_i = 4, \quad 
				F_i \cdot G_j = F_i \cdot F_j = G_i \cdot 
				G_j = 2 \quad
				\textup{ for all } i\neq j.
			\]
			These relations imply that $G_1, F_1, F_2, \ldots, F_{14}$ generate 
			a rank 15 sublattice of $\Pic\Yabcbar$.  We have listed equations 
			for representatives of these classes and the Galois action of 
			$G_{\Q}$ on them in Appendix~\ref{app:splittingfield}. We will use 
			this information in later sections.

			\subsubsection{Fibrations that descend to $\Xabcbar$}
				We work in characteristic zero for the remainder of this 
				section. Let $f\colon Y \to X$ be the K3 double-cover of a 
				generic Enriques surface $X$; assume that $Y$ is of the 
				form~\eqref{eq: general double cover}. In this case, the 
				degeneracy locus $Z$ contains at least one singular point $P$; 
				let $\phi_P, \phi'_P\colon Y\to \PP^1$ be the fibrations 
				described above. If there is an involution $\iota\colon\PP^1 \to 
				\PP^1$ such that $\phi\circ\sigma = \iota\circ\phi$, then each 
				fibration descends to $X$, i.e., there is a genus $1$ fibration 
				$\phi_X\colon X \to \PP^1$ such that the diagram
				\[
				 	\minCDarrowwidth55pt
					\begin{CD}
						Y @>>f> Y/\sigma = X \\
						@VV\phi_P V\phantom{Y/\sigma=} @VV\phi_X V\\
						\PP^1 @>>> \PP^1/\iota = \PP^1
					\end{CD}
				\]
				commutes, and similarly for $\phi_P'$.

				Careful inspection of the singular locus of $Z$ allows us to 
				determine which points $P$ have an associated involution $\iota$ 
				as above. Indeed, since the quadrics defining $Y$ are 
				differences of a quadric in $k[x,y,z]$ and a quadric in 
				$k[s,t,u]$, the sextic curve defining $Z \subset \PP^2$ is the 
				union of two (possibly reducible) cubic curves. Generically, 
				these two cubics intersect in nine distinct points. It is 
				precisely these singular points of $Z$ that have an associated 
				involution $\iota$ such that $\phi\circ\sigma = \iota\circ\phi$.

				The map $\PP^1 \to  \PP^1/\iota$ is ramified above the fixed 
				points of $\iota$.  Since  $f\colon Y\to X$ is unramified, the 
				morphism $\phi_X\colon X\to \PP^1$ must have non-reduced	
				fibers above the fixed points of $\iota$.  We denote by $C_P$ 
				and $\Ctilde_P$ (resp.\ $D_P$ and $\Dtilde_P$) the reduced 
				subschemes of the nonreduced fibers $F_P$ (resp.\ $G_P$) of 
				$\phi_P$.

				Let us specialize to our particular Enriques surface $\Xabcbar$ 
				and its K3 double cover $\Yabcbar$.  We already know that $Z$ 
				contains $14$ singular points $P_1, \ldots, P_{14}$.  We may 
				renumber these points so that $P_1, \ldots, P_9$ correspond to 
				the fibrations that descend to $\Xabcbar$.  This gives us $36$ 
				curves $C_i, \Ctilde_i, D_i, \Dtilde_i, i = 1,\ldots, 9$ on 
				$\Xabcbar$.  After possibly interchanging $D_i$ and $\Dtilde_i$ 
				for some $i$, we have the linear equivalence relations
				\[
					C_i + D_i = \Ctilde_j + \Dtilde_j, 
					\quad 2(C_i - \Ctilde_i) = 2(D_i - \Dtilde_i) = 0,
					\quad f^*C_i = f^*\Ctilde_i = F_i, 
					\quad f^*D_i = f^*\Dtilde_i = G_i, 
				\]
				for all $i,j$. Combining the projection formula with the 
				intersection numbers on $\Yabcbar$, we obtain
				\[
					C_i^2 = D_i^2 = 0, \quad 
					C_i\cdot D_j = C_i \cdot C_j = D_i \cdot D_j = 1, \quad
					C_i\cdot D_i = 2, \quad \textup{ for all }i\neq j.
				\]
				We have listed the action of the Galois group $G_{\Q}$ on $C_i, 
				D_i$ in Appendix~\ref{app:splittingfield}.

			\subsubsection{Splitting field of the genus $1$ fibrations}
			\label{subsec:galoisgeneral}
				Let $a$, $b$ and $c$ be indeterminates and $\mathbf{a} = 
				(a,b,c)$; consider $\Yabc$ and $\Xabc$ as surfaces over 
				$\Q(a,b,c)$.  The splitting field $K$ of the fibers of all the 
				genus $1$ fibrations is a degree $2^{18}$ extension (explicit 
				generators can be found in Appendix~\ref{app:splittingfield}).  
				We can consider $K$ as a $2^{18}$-cover of $\Aff^3$.  We say 
				that the triplet $\mathbf{a_0} := (a_0,b_0,c_0) \in \Z^3_{>0}$ 
				is \defi{Galois general} if the special fiber 
				$K_{(a_0,b_0,c_0)}$ is a field, i.e., if the splitting field of 
				the fibers of the genus $1$ fibrations of $Y_{\mathbf{a_0}}$ is 
				a degree $2^{18}$ extension over $\Q$. 

		\subsection{Upper bounds for $\rho(\Yabc)$ (after van Luijk)}
		\label{subsec:rank}

			Let $p \in \Z$ be a prime of good reduction for $\Yabc$, and write 
			$\Yabcp$ for the mod~$p$ reduction of $\Yabc$.  Suppose that $\Yabc$ 
			has good reduction at two distinct primes $p_1$ and $p_2$, that 
			$\rho(\Yabcpi) \leq n$ for each $i$, and that the discriminants of 
			the Picard groups of the reductions lie in different square classes. 
			Then $\rho(\Yabc) \leq n-1$ (see~\cite{vanLuijk-pic1}*{Proof of 
			Theorem~3.1}).  

			Let $\ell\neq p$ be a prime, and write $\psi_p(T)$ for the 
			characteristic polynomial of the action of Frobenius on 
			$\HH^2_{\et}\big(\Yabcp,\Q_{\ell}\big)$. Then $\rho(\Yabcp)$ is 
			bounded above by the number of roots of $\psi_p$ (counted with 
			multiplicity) that are of the form $p\zeta$, where $\zeta$ is a root 
			of unity~\cite[Corollary 2.3]{vanLuijk-pic1}. Using the Lefschetz 
			trace formula and Newton's identities, we may compute the 
			coefficient of $T^i$ in $\psi_p(T)$ in terms of $\#\Yabcp(\F_{p}), 
			\ldots, \#\Yabcp(\F_{p^{22-i}})$.  The functional equation
			\begin{equation}
				\label{eq: functional equation}
				p^{22}\psi_p(T) = \pm T^{22}\psi_p(p^2/T),
			\end{equation}
			then allows us to compute the coefficient of $T^{23-i}$, up to sign.  
			In addition, for a subgroup $M \subseteq \Pic \Yabcpbar$, the 
			polynomial $\psi_M(T/p)$ divides $\psi_p(T)$, where $\psi_M(T)$ is 
			the characteristic polynomial of Frobenius acting on $M$.  Thus, 
			knowing the action of Frobenius on a rank $r$ subgroup of $\Pic 
			\Yabcpbar$, together with $\#\Yabcp(\F_{p}), \ldots, 
			\#\Yabcp(\F_{p^{\lceil(22-r)/2\rceil}})$, allows us to compute up to 
			two possible $\psi_p(T)$'s, each corresponding to a choice of sign 
			in~\eqref{eq: functional equation}.  In some cases, this is enough 
			information to rule out one of the sign choices; for more ways to 
			determine the sign choice, see~\cite{EJ-signs}.

			\begin{proposition}
				\label{prop: Pic upper bound}
				Let $\mathbf{a} \in \Z^3_{>0}$ satisfy conditions $(5)$ and 
				$(6)$. Then $\rho(\Yabc) \leq 15$.
			\end{proposition}

			\begin{proof}
				Let $M_p$ denote the subgroup of $\Pic \Yabcpbar$ generated by 
				$G_1,F_1,\dots,F_{14}$.  Let $\widetilde{\psi}_{p}(T) := 
				p^{-22}{\psi_p}(pT)$, so that the number of roots of 
				$\widetilde{\psi}_{p}(T)$ that are roots of unity gives an upper 
				bound for the geometric Picard number of the reduced surface. As 
				described above, computing the action of Frobenius on $M_p$ and 
				computing $\#\Yabcp(\F_{p^{i}})$ for $i := 1, 2,3,4$ is enough 
				to determine $\tilde{\psi}_{p}(T)$ for $p=7$ and $11$: 
				\begin{align*}
					\widetilde{\psi}_{7}(T) &= \frac{1}{7}(T - 1)^8(T + 1)^8
					(7T^6 + 6T^5 + 9T^4 + 4T^3 + 9T^2 + 6T + 7), 
					\\
					\widetilde{\psi}_{11}(T) &= 
					\frac1{11}(T - 1)^8(T + 1)^4(T^2 + 1)^2
					(11T^6 - 2T^5 + T^4 + 12T^3 + T^2 - 2T + 11).
				\end{align*}
				In both cases, the roots of the degree $6$ factor of 
				$\widetilde{\psi}_{p}(T)$ are not integral, so they are not 
				roots of unity.  We conclude that $\rho(\Yabcp) \leq 16$ for 
				$p = 7$ and $11$.


				Next, we compute the square class of the discriminant $\Delta$ 
				of each reduced Picard lattice via the Artin-Tate conjecture, 
				which is known to hold for K3 surfaces endowed with a genus $1$ 
				fibration (see~\cites{Artin-SwinnertonDyer,Milne-AT}):
				\[
					\lim_{T \to p} \frac{\psi_p(T)}{(T - p)^{\rk(\Pic\Yabcp)}} = 
					p^{21 - \rk(\Pic\Yabcp)}\cdot\#\Br(\Yabcp)\cdot|\Delta|
				\]
				Observe that $\#\Br(\Yabcp)$ is always a square~\cite{LLR}. 
				Write $\big[ |\Delta| \big]$ for the class of $|\Delta|$ in 
				$\Q^{\times}/\Q^{\times2}$. We compute
				\[
					\big[ |\Delta(\Yabcsevenbar)| \big] = 3
					\quad\text{and}\quad
					\big[ |\Delta(\Yabcelevenbar)| \big] = 2,
				\]
				and thus $\rho(\Yabc) \leq 15$, completing the proof.
			\end{proof}

			\subsection{Determining $\Pic\Yabcbar$}\label{subsec:saturation}
				The sublattice $L := \langle G_1,F_1,\dots,F_{14} \rangle$ of 
				$\Pic \Yabcbar$ has discriminant $2^{17}$.  In this subsection 
				we determine its saturation in $\Pic \Yabcbar$.

				Each line $\ell$ on the curve $Z$ corresponds to a pencil of 
				quadrics; the vanishing locus of this pencil $S_\ell$ defines a 
				del Pezzo surface of degree $4$, embedded in a hyperplane in 
				$\PP^5$.  The inclusion of the pencil in the net of quadrics 
				defining $\Yabcbar$ gives a morphism $\Yabcbar \to S_\ell$.  
				Recall that $Z$ contains four lines, so we obtain four such 
				maps.  Pulling back the exceptional curves on each of the del 
				Pezzo surfaces, we obtain $16$ order $2$ elements in $(\Pic 
				\Yabcbar)/L$, generated by 
				\begin{align}
					\frac{1}{2}(F_1 + F_2 + F_3 + F_{10} + F_{12}),\quad &
					\frac{1}{2}(F_1 + G_1 + F_4 + F_5 + F_6 + F_{10} + F_{11}),
					\label{eqn:lines1}\\
					\frac{1}{2}(F_1 + F_4 + F_7 + F_{13} + F_{14}), \quad &
					\frac{1}{2}(F_1 + G_1 + F_7 + F_8 + F_9 + F_{11} + F_{12}).
					\label{eqn:lines2}
				\end{align}

				We owe the idea behind the proof of the following proposition to 
				Michael Stoll and Damiano Testa~\cite{ST-rationalbox}.
				\begin{proposition}
					\label{prop:saturation}
					Let $\mathbf{a} \in \Z^3_{>0}$ satisfy condition $(8)$. The 
					sublattice $L' \subseteq \Pic\Yabcbar$ spanned by 
					$G_1,F_1,\dots,F_{14}$ and the classes in~\eqref{eqn:lines1} 
					and~\eqref{eqn:lines2} is saturated.
				\end{proposition}

				\begin{proof}
					The lattice $L'$ has discriminant $2^9$, so it suffices to 
					show that the induced map
					\[
						\phi\colon L'/2L' \to \Pic\Yabcbar/2\Pic\Yabcbar
					\]
					is injective.  The action of $G_{\Q}$ on $L'$ factors 
					through a finite group $G$ whose order divides $2^{18}$ (see 
					\S\ref{subsec:galoisgeneral}). Consider the induced 
					$G$-equivariant homomorphism
					\[
						\phi^{G}\colon \big(L'/2L'\big)^{G} \to 
						\big(\Pic\Yabcbar/2\Pic\Yabcbar\big)^{G},
					\]
					and note that if $\ker\phi$ is nonzero then $(\ker \phi)^{G} 
					= \ker \phi^{G}$ is also nonzero, because any representation 
					of a $2$-group by a nonzero $\F_2$-vector space has a 
					nonzero invariant 
					subspace~\cite{Serre-LinReps}*{Proposition~26}.  Using 
					condition $(8)$, it is easy to establish that $ 
					\big(L'/2L'\big)^{G}$ is a $2$-dimensional $\F_2$-vector 
					space, spanned by the classes $v_1 := [G_1 + F_{10}]$ and 
					$v_2 := [F_2 + F_3 + F_{12}]$.  If $\phi^G(v_1) = 0$ then 
					$G_1 + F_{10} \in 2\Pic\Yabcbar$; however, the intersection 
					pairing on $\Pic\Yabcbar$ is even, and $\frac{1}{2}(G_1 + 
					F_{10})\cdot\frac{1}{2}(G_1 + F_{10}) = 1$. Hence 
					$\phi^G(v_1) \neq 0$.  Similarly, $\phi^G(v_2), \phi^G(v_1 + 
					v_2) \neq 0$, and we conclude that $\ker\phi^G = 0$.
				\end{proof}

				\begin{cor}\label{cor:PicK3}
					Let $\mathbf{a}\in \Z^3_{>0}$ satisfy conditions $(5)$, 
					$(6)$, and $(8)$.  Then $L'  = \Pic\Yabcbar$.	
				\end{cor}

				\begin{proof}
					By Proposition~\ref{prop: Pic upper bound} we have 
					$\rho(\Yabc) \leq 15$.  Thus $L'$ has full rank inside 
					$\Pic\Yabcbar$, and it is saturated by 
					Proposition~\ref{prop:saturation}.
				\end{proof}

			\subsection{Determining $\Num\Xabcbar$}\label{subsec:Enriques}

				Let $M := \langle D_1, C_1, C_2, \ldots, C_9\rangle \subseteq 
				\Num \Xabcbar$.  In \S\ref{subsec:fibrations}, using 
				intersection numbers, we calculated that $\rk M = 10$, and that 
				the discriminant of the intersection lattice is $4$.  Since 
				$\Num \Xabcbar$ is a rank $10$ unimodular lattice 
				(see~\S\ref{subsec:PicEnr}), $M$ is an index $2$ subgroup of 
				$\Num \Xabcbar$.  The following proposition shows that there are 
				only two possible saturations of $M$.

				\begin{proposition}
					\label{prop: NumXbar}
					There exists a divisor $R$ on $\Xabcbar$ such that $2R$ is 
					linearly equivalent to either
					\begin{equation}
						\label{eq: possible Rs in Num}
						C_1 + C_2 + \cdots + C_9 \quad\text{or}\quad
						D_1 + C_2 + \cdots + C_9,
					\end{equation}
					 and such that $\Num\Xabcbar = \langle R, D_1, C_1, \ldots, 
					C_9\rangle$.
				\end{proposition}

				\begin{proof}
					From the discussion above, we know that $\Num \Xabcbar/M 
					\isom \Z/2\Z$.  Let $R$ be a divisor whose class in 
					$\Num\Xabcbar$ is not in $M$. Without loss of generality, we 
					may assume that 
					\[
						2R = n_1C_1 + \cdots n_9C_9 + n_{10}D_1
					\] 
					in $\Pic\Xabcbar$, where each $n_i \in \{0,1\}$.  Since $R$ 
					pairs integrally with $D_1$ and $C_i$ for all $i$ and $R^2 
					\equiv 0 \pmod 2$, we must have that $n_2 \equiv n_3\equiv 
					\cdots n_9\equiv 1\pmod 2$ and $n_1 + n_{10} \equiv 1 \pmod 
					2$, giving the desired result.
				\end{proof}

			\subsection{Galois cohomology}\label{subsec:cohom}

				Throughout this section, we assume that $\mathbf{a} \in 
				\Z^3_{>0}$ is Galois general, and we let $\Kabc$ be the 
				splitting field described in \S\ref{subsec:galoisgeneral}.  
				Since $\Pic \Yabcbar$ and $\Num \Xabcbar$ are torsion-free, to 
				compute the Galois cohomology of $\Pic \Yabcbar$ and $\Num 
				\Xabcbar$ it suffices to compute the action of $G_{\Q}$ on the 
				curves $C_i, D_i, F_i, G_i$.  Thus, the action of $G_{\Q}$ on 
				$\Pic\Yabcbar$ and $\Num\Xabcbar$ factors through 
				$\Gal(\Kabc/\Q)$. The action of $\Gal(\Kabc/\Q)$ is described in 
				Table~\ref{table:galois}.  Together with the generators for 
				$\Pic \Yabcbar$ and $\Num \Xabcbar$ given in 
				\S\S\ref{subsec:saturation}--\ref{subsec:Enriques}, it allows us 
				to prove the following proposition.
				\begin{prop}
					\label{prop: Galois Cohomology}
					Let $\mathbf{a} \in \Z^3_{>0}$ satisfy condition $(8)$. Then 
					$\HH^1(G_{\Q}, \Num \Xabcbar) = \{1\}$.
					\qed
				\end{prop}
				We used \texttt{Magma}~\cite{magma} for the computations 
				involved in Proposition~\ref{prop: Galois Cohomology}, but the 
				industrious reader may verify the results by hand, using the 
				following lemma.
				\begin{lemma}
					Let $G$ be a (pro-)finite group and $A$ be a (continuous) 
					torsion-free $G$-module.  Then
					\[
						\HH^1(G, A)[m] \isom 
						\frac{\left(A/mA\right)^G}{A^G/mA^G}.
						\eqno{\qed}
					\]
				\end{lemma}
				Proposition~\ref{prop: Galois Cohomology} implies that the long 
				exact sequence in group cohomology associated 
				to~\eqref{eqn:PicEnr} is given by
				\[
					0 \to \langle K_{\Xabc} \rangle \to \Pic \Xabc \to 
					\Num \Xabc \to \HH^1(G_{\Q},\langle K_{\Xabc}\rangle) \to 
					\HH^1(G_{\Q}, \Pic \Xabcbar) \to 0.
				\]
				The explicit description of the Galois action given in 
				Table~\ref{table:galois} allows us to verify that $\Pic \Xabc = 
				K_{\Xabc} \oplus \Z\cdot[C_1 + D_1]$ and that $\Num \Xabc = 
				\Z\cdot[C_1 + D_1]$.  The following proposition now follows 
				easily.
				\begin{prop}\label{prop:H1PicEnr}
					Let $\mathbf{a} \in \Z^3_{>0}$ satisfy condition $(8)$. Then 
					\[
						\HH^1(G_{\Q}, \Pic \Xabcbar) \isom 
						\HH^1(G_{\Q}, \langle K_{\Xabc}\rangle).\eqno{\qed}
					\]
				\end{prop}

				\begin{remark}
					Table~\ref{table:galois} also allows us to show that if 
					$\mathbf{a}$ satisfies conditions $(5),(6)$ and $(8)$ then 
					$\HH^1(G_{\Q}, \Pic \Yabcbar) \isom \Z/2\Z$. While this 
					result is not logically necessary for our purposes, it is 
					worth noting that it implies that the algebra $h^*(\calA)$ 
					in the proof of Proposition~\ref{prop:noQpoints} gives a 
					representative of the nontrivial class of this cohomology 
					group, and thus, up to adjustment by constant Azumaya 
					algebras arising from the base field, $h^*(\calA)$ is the 
					\emph{only} algebra that can give an algebraic Brauer-Manin 
					obstruction to the Hasse principle on $\Yabc$.
				\end{remark}

	\section{The Brauer-Manin obstruction}\label{sec:BrSet}

		\begin{prop}\label{prop:algbrauerset}
			Let $\mathbf{a} \in \Z^3_{>0}$ satisfy conditions $(7)$ and $(8)$. 
			Then $\Xabc(\Adeles_{\Q})^{\Br_1} \neq \emptyset$.
		\end{prop}
		\begin{proof}
			Since $\mathbf{a}$ satisfies condition $(8)$, by 
			Proposition~\ref{prop:H1PicEnr} we know that $\HH^1(G_{\Q}, \Pic 
			\Xabcbar) \isom \HH^1(G_{\Q}, \langle K_{\Xabc} \rangle)$.  Using 
			this in conjunction with~\cite[Theorem 6.1.2]{Skorobogatov-torsors} 
			and the isomorphism coming from the Hochschild-Serre spectral 
			sequence
			\[
				\Br_1 X/\im \Br \Q \xrightarrow{\sim} \HH^1(G_{\Q}, \Pic \Xbar),
			\]
			we obtain the following partition of the algebraic Brauer set
			\begin{align}\label{eqn:partition}
				\Xabc(\Adeles_{\Q})^{\Br_1} = \bigcup_{\tau \in \HH^1(G_{\Q}, 
				\langle K_{\Xabc} \rangle)} 
				f^{\tau}\left(\Yabc^{\tau}(\Adeles_{\Q})\right).
			\end{align}
			By condition (7) we have $\Yabc(\Adeles_{\Q}) \neq \emptyset$, and 
			thus $\Xabc(\Adeles_{\Q})^{\Br_1} \neq \emptyset$. 
		\end{proof}

		\begin{prop}\label{prop:moreover}
			Let $\mathbf{a} \in \Z^3_{>0}$ satisfy conditions $(5)$--\,$(8)$, 
			and assume that $\Xabc(\Adeles_{\Q})^{\Br} = \emptyset$. Then there 
			exists an element $\calA \in \Br Y \setminus\Br_1 Y$ such that 
			$\Yabc(\Adeles_{\Q})^{\calA} = \emptyset$.
		\end{prop}
		\begin{proof}
			Since $\Br \Xbar \isom \Z/2\Z$ for \emph{any}  Enriques surface 
			$X$~\cite[p.\ 3223]{HS-Enriques}, there is at most one nontrivial 
			class in $\Br X/\Br_1 X$.  If $\Br \Xabc = \Br_1 \Xabc$, then 
			$\Xabc(\Adeles_{\Q})^{\Br} \neq \emptyset$ (by 
			Proposition~\ref{prop:algbrauerset}), contradicting our hypotheses.  
			Thus we may assume that there exists an element $\calA' \in \Br 
			\Xabc \setminus \Br_1 \Xabc$.  Then $\Xabc(\Adeles_{\Q})^{\Br} = 
			\Xabc(\Adeles_{\Q})^{\Br_1} \cap \Xabc(\Adeles_{\Q})^{\calA'}$, and 
			\eqref{eqn:partition}, together with functoriality of the Brauer 
			group imply that
			\[
				\Xabc(\Adeles_{\Q})^{\Br} = \bigcup_{\tau \in \HH^1(G_{\Q}, 
				\langle K_{\Xabc} \rangle)} 
		f^{\tau}\left(\Yabc^{\tau}(\Adeles_{\Q})^{(f^{\tau})^*\calA'}\right).
			\]
			By assumption, we have $\Xabc(\Adeles_{\Q})^{\Br} = \emptyset$, so 
			it suffices to show that $f^*\calA'\notin\Br_1 \Yabc$, because then 
			we can take $\calA = f^*\calA'$.  Equivalently, we show that the map 
			$f^*\colon\Br \Xabcbar \to \Br \Yabcbar$ is injective, using 
			Beauville's criterion~\cite[Cor. 5.7]{Beauville-BrEnr}: $f^*$ is 
			injective if and only if there is no divisor class $D\in \Pic 
			\Yabcbar$ such that $D^2 \equiv 2 \pmod 4$ and $\sigma(D) = -D$.  
			This is a straightforward computation, using the integral basis of 
			$\Pic \Yabcbar$ given in Corollary~\ref{cor:PicK3}
		\end{proof}

	\section{Proof of Theorems~\ref{thm:main-abc} and~\ref{thm:main}}
	\label{sec:proofofprops}

		\begin{lemma}\label{lem:localpoints}
			$Y_{(12, 111, 13)}(\Adeles_{\Q}) \neq \emptyset$.
		\end{lemma}
		\begin{proof}
			By the Weil conjectures, if $p \geq 22$ is a prime of good 
			reduction, then $Y_{(12,111,13)}$ has $\F_p$-points.  Hensel's lemma 
			then implies that $Y_{(12,111,13)}(\Q_p)\neq\emptyset$ for all such 
			primes. It remains to consider the infinite place, the primes less 

			than $22$, and the primes of bad reduction: 
			$2,3,5,13,37,59,151,157,179,821,881$ and $1433$. It is easy to check 
			that $Y_{(12,111,13)}(\R) \neq \emptyset$, and that for a finite 
			prime $p>5$ in our list, $Y_{(12,111,13)}$ has a smooth 
			$\F_p$-point.  For the remaining three primes we exhibit explicit 
			local points:
			\begin{align*}
				(\sqrt{129}:2\sqrt{21/5}:\sqrt{2113}:1:4:5) 
					& \in Y_{(12,111,13)}(\Q_2),\\
				(0:0:\sqrt{821/5}:-2:1:\sqrt{2/5}) &\in  
					Y_{(12,111,13)}(\Q_3),\\
				(1:2\sqrt{-1}:\sqrt{136}:1:-4:1) &\in  Y_{(12,111,13)}(\Q_5). 
				\tag*{\qed}
			\end{align*}
				\hideqed
		\end{proof}

			\begin{proof}[Proof of Theorem~\ref{thm:main-abc}]
				Combining Propositions~\ref{prop:noQpoints} 
				and~\ref{prop:algbrauerset} we obtain the first statement.  The 
				second statement follows from Proposition~\ref{prop:moreover}.
			\end{proof}	

			\begin{proof}[Proof of Theorem~\ref{thm:main}]
				The triplet $\mathbf{a} = (12,111,13)$ satisfies the eight 
				required conditions: $(1)-(6)$ and $(8)$ are easily checked, and 
				$(7)$ follows from Lemma~\ref{lem:localpoints}.
			\end{proof}

			\subsection{More possibilities}
				Although we have shown only that the triplet 
				$\mathbf{a} = (12,111,13)$ satisfies conditions $(1)$--$(8)$, 
				there are in fact many triplets that do the job.  Additionally, 
				we can weaken conditions $(5)$ and $(6)$, and thus find even 
				more triplets.  

				We used conditions $(5)$ and $(6)$ to ensure that 
				$\rho(\Yabcseven) = \rho(\Yabceleven) = 16$ and that the 
				discriminants of the respective intersection lattices did not 
				differ by a square.  Testing all possible isomorphism classes of 
				$\Yabc$ modulo $7$ and modulo $11$, we found $10\,380$ possible 
				congruence classes of $\mathbf{a}$ modulo $77$ that have the 
				desired property concerning Picard numbers and discriminants.

				A computer search shows that there are $202$ triplets 
				$\mathbf{a}\in\Z^3_{>0}$ satisfying conditions $(1)$--$(4)$ and 
				$(7)$ and these weaker versions of $(5)$ and $(6)$ with $a + b + 
				c \leq 500$.  We do not know how to verify condition $(8)$ 
				efficiently on a computer.  However, Hilbert's irreducibility 
				theorem says we should expect condition $(8)$ to hold ``almost 
				always'', so we feel this data does give some evidence that 
				several triplets $\mathbf{a}$ satisfying conditions $(1)$--$(8)$ 
				exist.

			\vspace{-.05in}
			\subsection{Varying the del Pezzo surface}\label{subsec:dp4details}%
				The Enriques surface $\Xabc$ was built up from the degree $4$ 
				del Pezzo surface $S$ given by~\eqref{eq:dp4-1} 
				and~\eqref{eq:dp4-2}.  The equality $\HH^1(G_{\Q}, \Num 
				\Xabcbar) = \{1\}$ depends on the interaction of $\sigma$ with 
				$S$ in a way that is perhaps not obvious in the proof.  More 
				precisely, consider the following K3 surfaces
				\[
					\begin{array}{lrllrl}
						& xy & = z^2 - 5s^2 & & xy & = s^2 - 5u^2\\
						Y_1: & (x + y)(x + 2y) & = z^2 - 5t^2 & 
						Y_2: & (x + y)(x + 2y) & = s^2 - 5t^2\\
						& Q_1(x,y,z) & = \Qtilde_1(s,t,u) & 
						& Q_2(x,y,z) & = \Qtilde_2(s,t,u)
					\end{array}
				\]
				where $Q_i, \Qtilde_i$ are any quadrics such that 
				$\sigma|_{Y_i}$ has no fixed points, and let $X_i := Y_i/\sigma, 
				i = 1,2$, be the corresponding Enriques surfaces.  Note that 
				each $Y_i$ maps to a del Pezzo surface which is $\Q$-isomorphic 
				to $S$.  In this case $\#\HH^1(G_{\Q}, \Num \Xbar_i)\geq 2$ 
				regardless of the choice of $Q_i, \Qtilde_i$.  In fact, the 
				quaternion algebras $\left(\frac{5z^2}{s^2}, 
				\frac{x+y}{x}\right) \in \Br \kk(X_1)$ and $\left(5, 
				\frac{x+y}{x}\right) \in \Br \kk(X_2)$ lift to non-trivial 
				elements in $\Br X_1$ and $\Br X_2$, respectively.
				\vspace{-.05in}
				\begin{remark}
					Although $\frac{z^2}{s^2}$ is a square in $\kk(Y_2)$, it is 
					\emph{not} a square in $\kk(X_2)$, so 
					$\left(\frac{5z^2}{s^2}, \frac{x+y}{x}\right) \in \Br 
					\kk(X_1)$ is not necessarily equal to $\left(5, 
					\frac{x+y}{x}\right) \in \Br \kk(X_1)$.
				\end{remark}
				\vspace{-.05in}
				\begin{remark}
					One can check that, up to a $\Q$-automorphism of $\PP^5$ 
					that commutes with $\sigma$, $Y_1$ and $Y_2$ are the only K3 
					surfaces that arise as double covers of some Enriques 
					surface and that map to a del Pezzo surface which is 
					$\Q$-isomorphic, but not equal, to $S$.
				\end{remark}

	\appendix
	\vspace{-.125in}
	\section{The splitting field of fibers of genus $1$ fibrations}
	\label{app:splittingfield}
		The splitting field $K$ of the genus $1$ curves $C_i, \Ctilde_i, 
		\Dtilde_i, F_i, G_i$ is generated by 
		\begin{align}
			i, \sqrt2, \sqrt5, \sqrt{a}, \sqrt{c}, \sqrt{c^2 - 100ab}, 
			\gamma:= \sqrt{-c^2 - 5bc - 10ac - 25ab},\label{eqn:gens1}\\
			 \sqrt[4]{ab}, \sqrt{-2 + 2\sqrt{2}},
			\sqrt{-c-10\sqrt{ab}}, \label{eqn:gens2}\\
			\theta_0 := \sqrt{4a^2 + b^2},\; \delta_1 := \sqrt{a + b + c/5}, \;	
			\delta_2:= \sqrt{a + b/4 + c/10}, \label{eqn:gens3}\\
			\theta_1^+ := \sqrt{20a^2 -10ab - 2bc + (10a + 2c)\theta_0},\; 
			\theta_2^+ := \sqrt{-5a - 5/2b - 5/2\theta_0}, \label{eqn:gens4}\\
			\xi_1^+ := \sqrt{20a + 10b + 3c + 20\delta_1\delta_2},\;
			\xi_2^+ := \sqrt{4a + 2b + 2/5c + 4\delta_1\delta_2}.
			\label{eqn:gens5}
		\end{align}
		The field extension $K/\Q$ is Galois, as the following relations show:
		\begin{align*}
			\sqrt{-2 - 2\sqrt{2}} = \frac{2i}{\sqrt{-2 + 2\sqrt{2}}},\quad &
			\sqrt{-c + 10\sqrt{ab}} = 
				\frac{\sqrt{c^2 - 100ab}}{\sqrt{-c - 10\sqrt{ab}}}\\
			\theta_1^- := \sqrt{20a^2 -10ab - 2bc + (10a + 2c)\theta_0} 
			= \frac{4a\gamma}{\theta_1^+}, \quad&
			 \theta_2^- := \sqrt{-5a - 5/2b + 5/2\theta_0} 
			= \frac{5\sqrt{ab}}{\theta_2^+}\\
			\xi_1^- := \sqrt{20a + 10b + 3c - 20\delta_1\delta_2} 
			= \frac{\sqrt{c^2 - 100ab}}{\xi_1^+}, \quad &		
			\xi_2^- := \sqrt{4a + 2b + 2/5c - 4\delta_1\delta_2} 
			= \frac{2\gamma}{5\xi_2^+}
		\end{align*}

		Table~\ref{table:curves} lists two linear forms $\ell_1, \ell_2$ next to 
		each divisor class.  On the K3 surface, the vanishing of these linear 
		forms defines a curve representing the corresponding class.  On the 
		Enriques surface, the curve $f(\Yabcbar \cap V(\ell_1, \ell_2))_{\red}$ 
		represents the given divisor class.
		

		\begin{landscape}
		\begin{table}[!p]
			\begin{tabular}{|c|c|rl||c|c|rl|}
				\hline
				$\Yabcbar$ & $\Xabcbar$ & Defining & equations &
				$\Yabcbar$ & $\Xabcbar$ & Defining & equations \\
				\hline
				$F_1$ & $C_1$ & $s - \sqrt{5}t,$ & $x + 2y$&
				$F_4$ & $C_4$ & $\sqrt{c}s - \sqrt5u,$ & $
								 10ax - (c + \sqrt{c^2-100ab})y$\\

				& $\Ctilde_1$ & $s + \sqrt{5}t,$ & $ x + y$&
				& $\Ctilde_4$ & $\sqrt{c}s + \sqrt5u,$ & $
								 10ax - (c - \sqrt{c^2-100ab})y$\\
				$G_1$ & $D_1$ & $s - \sqrt{5}t,$ & $ x + y$ &
				$G_4$ & $D_4$ & $\sqrt{c}s - \sqrt5u,$ & $
								 10ax - (c - \sqrt{c^2-100ab})y$\\
				& $\Dtilde_1$ & $s + \sqrt{5}t,$ & $ x + 2y$&
				& $\Dtilde_4$ & $\sqrt{c}s + \sqrt5u,$ & $
								 10ax - (c + \sqrt{c^2-100ab})y$\\
				\hline
				$F_2$ & $C_2$ & $\sqrt{-2 + 2\sqrt{2}}s - \sqrt5t,$ & $ 
								 x + \sqrt{2}y + \sqrt{5}(1 - \sqrt{2})z$&
				$F_5$ & $C_5$ & $i\sqrt2\sqrt[4]{ab}s - u,$ & 
							$\sqrt{a}x + \sqrt{b}y + \sqrt{-c - 10\sqrt{ab}}z$\\
				& $\Ctilde_2$ & $\sqrt{-2 + 2\sqrt{2}}s + \sqrt5t,$ & $ 
								 x + \sqrt{2}y - \sqrt{5}(1 - \sqrt{2})z$ &
				& $\Ctilde_5$ & $i\sqrt2\sqrt[4]{ab}s + u,$ & 
							$\sqrt{a}x + \sqrt{b}y - \sqrt{-c - 10\sqrt{ab}}z$\\

				$G_2$ & $D_2$ & $\sqrt{-2 + 2\sqrt{2}}s - \sqrt5t,$ & $ 
								 x + \sqrt{2}y - \sqrt{5}(1 - \sqrt{2})z$&
				$G_5$ & $D_5$ & $i\sqrt2\sqrt[4]{ab}s - u,$ & 
				 			$\sqrt{a}x + \sqrt{b}y - \sqrt{-c - 10\sqrt{ab}}z$\\				

				& $\Dtilde_2$ & $\sqrt{-2 + 2\sqrt{2}}s + \sqrt5t,$ & $ 
								 x + \sqrt{2}y + \sqrt{5}(1 - \sqrt{2})z$&
				& $\Dtilde_5$ & $i\sqrt2\sqrt[4]{ab}s + u,$ & 
							$\sqrt{a}x + \sqrt{b}y + \sqrt{-c - 10\sqrt{ab}}z$\\

				\hline
				$F_3$ & $C_3$ & $\sqrt{-2 - 2\sqrt{2}}s - \sqrt5t,$ & $
								 x - \sqrt2y + \sqrt5(1 + \sqrt2)z$&
				$F_6$ & $C_6$ & $\sqrt2\sqrt[4]{ab}s + u,$ & 
							$\sqrt{a}x - \sqrt{b}y + \sqrt{-c + 10\sqrt{ab}}z$\\
				& $\Ctilde_3$ & $\sqrt{-2 - 2\sqrt{2}}s + \sqrt5t,$ & $
								 x - \sqrt2y - \sqrt5(1 + \sqrt2)z$&
				& $\Ctilde_6$ & $\sqrt2\sqrt[4]{ab}s - u,$ & 
							$\sqrt{a}x - \sqrt{b}y - \sqrt{-c + 10\sqrt{ab}}z$\\
				$G_3$ & $D_3$ & $\sqrt{-2 - 2\sqrt{2}}s - \sqrt5t,$ & $
								 x - \sqrt2y - \sqrt5(1 + \sqrt2)z$ &
				$G_6$ & $D_6$ & $\sqrt2\sqrt[4]{ab}s + u,$ & 
							$\sqrt{a}x - \sqrt{b}y - \sqrt{-c + 10\sqrt{ab}}z$\\
				& $\Dtilde_3$ & $\sqrt{-2 - 2\sqrt{2}}s + \sqrt5t,$ & $
								 x - \sqrt2y + \sqrt5(1 + \sqrt2)z$ &
				& $\Dtilde_6$ & $\sqrt2\sqrt[4]{ab}s - u,$ & 
							$\sqrt{a}x - \sqrt{b}y + \sqrt{-c + 10\sqrt{ab}}z$\\
				\hline
				\end{tabular}
				\label{table:curves2}
			\end{table}
			\begin{table}[!p]
				\begin{tabular}{|c|c|rl||c|c|rl|}
					\hline
					$\Yabcbar$ & $\Xabcbar$ & Defining & equations &
					$\Yabcbar$ & $\Xabcbar$ & Defining & equations \\
					\hline

				\hline
				$F_7$ & $C_7$ & $\sqrt{c}t - u,$ & $(5a + c)x + (c + \gamma)y$ &
				$F_{10}$ && $s - \sqrt5z,$ & $x$\\
				& $\Ctilde_7$ & $\sqrt{c}t + u,$ & $(5a + c)x + (c - \gamma)y$ &
				$G_{10}$ && $s - \sqrt5z,$ & $y$\\\cline{5-8}
				$G_7$ & $D_7$ & $\sqrt{c}t - u,$ & $(5a + c)x + (c - \gamma)y$ &
				$F_{11}$ && $u - \sqrt{c}z,$ & $\sqrt{a}x + i\sqrt{b}y$\\
				& $\Dtilde_7$ & $\sqrt{c}t + u,$ & $(5a + c)x + (c + \gamma)y$ &
				$G_{11}$ && $u - \sqrt{c}z,$ & $\sqrt{a}x - i\sqrt{b}y$\\
				\hline
				$F_8$ & $C_8$ & $\theta_2^+ t - u,$
				&$x + (2a + b + \theta_0)/(b + \theta_0)y - \theta_1^+/(2a)z$ &
				$F_{12}$ && $t - z,$ & $x + (1 - i)y$\\
				& $\Ctilde_8$ & $\theta_2^+ t + u,$
				&$x + (2a + b + \theta_0)/(b + \theta_0)y + \theta_1^+/(2a)z$ &
				$G_{12}$ && $t - z,$ & $x + (1 + i)y$\\\cline{5-8}
				$G_8$ & $D_8$ & $\theta_2^+ t - u,$
				&$x + (2a + b + \theta_0)/(b + \theta_0)y + \theta_1^+/(2a)z$ &
				$F_{13}$ && $\xi_2^+s - \xi_1^+t,$&
					$(\delta_1 + 2\delta_2)x + (2\delta_1 + 2\delta_2)y - u$\\
				& $\Dtilde_8$ & $\theta_2^+ t + u,$
				&$x + (2a + b + \theta_0)/(b + \theta_0)y - \theta_1^+/(2a)z$ &
				$G_{13}$ && $\xi_2^+s - \xi_1^+t,$&
					$(\delta_1 + 2\delta_2)x + (2\delta_1 + 2\delta_2)y + u$\\
				\hline
				$F_9$ & $C_9$ & $\theta_2^- t - u,$
				&$x + (2a + b - \theta_0)/(b - \theta_0)y - \theta_1^-/(2a)z$ &
				$F_{14}$ && $\xi_2^-s - \xi_1^-t,$&
					$(\delta_1 - 2\delta_2)x + (2\delta_1 - 2\delta_2)y - u$\\
				& $\Ctilde_9$ & $ \theta_2^- t + u,$
				&$x + (2a + b - \theta_0)/(b - \theta_0)y + \theta_1^-/(2a)z$ &
				$G_{14}$ && $\xi_2^-s - \xi_1^-t,$&
					$(\delta_1 - 2\delta_2)x + (2\delta_1 - 2\delta_2)y + u$\\
					\cline{5-8}
				$G_9$ & $D_9$ & $\theta_2^- t - u,$
				&$x + (2a + b - \theta_0)/(b - \theta_0)y + \theta_1^-/(2a)z$ 
				&&&&\\
				& $\Dtilde_9$ & $\theta_2^- t + u,$
				&$x + (2a + b - \theta_0)/(b - \theta_0)y - \theta_1^-/(2a)z$ 
				&&&&\\
				\hline
			\end{tabular}
			\caption{Defining equations of a curve representing a divisor class}
			\label{table:curves}
		\end{table}

	\begin{table}
		\begin{tabular}{|c|l|l||c|l|l||c|l|l|}
			\hline
			Action on & Action on & Action on & 
			Action on & Action on & Action on & 
			Action on & Action on & Action on \\
			splitting field & $\Pic\Xabcbar$ & $\Pic \Yabcbar$ &
			splitting field & $\Pic\Xabcbar$ & $\Pic \Yabcbar$ &
			splitting field & $\Pic\Xabcbar$ & $\Pic \Yabcbar$\\
			\hline
			$\sqrt5\mapsto-\sqrt5$ & $C_1 \mapsto D_1$ 
			& $F_1 \leftrightarrow G_1$
			& $i\mapsto -i$ & $C_3 \mapsto \Dtilde_3$ & $F_3 \mapsto G_3$
			& $\sqrt[4]{ab}\mapsto i\sqrt[4]{ab}$ & $C_5\mapsto C_6$ 
				& $F_5 \mapsto F_6$ \\

			& $D_1 \leftrightarrow \Ctilde_1$ & & & $C_5 \mapsto \Dtilde_5$ &
				$F_5 \mapsto G_5$
			& $\sqrt{- c - 10\sqrt{ab}} \leftrightarrow$ 
			& $C_6 \mapsto \Dtilde_5$ & $F_6 \mapsto G_5$\\

			& $C_2 \mapsto \Ctilde_2$ & & & & $F_{11} \mapsto G_{11}$ &
			$\sqrt{-c + 10\sqrt{ab}}$ & $C_9 \mapsto \Dtilde_9$ 
				& $F_9 \mapsto G_9$\\

			& $C_3 \mapsto \Ctilde_3$ & & & & $F_{12} \mapsto G_{12}$ & &&
				$F_{11} \mapsto G_{11}$\\

			& $C_4 \mapsto \Dtilde_4$ & $F_4 \mapsto G_4$ &&&&&&\\
			& & $F_{10} \mapsto G_{10}$ &&&&&& \\ 
			\hline

			$\sqrt{c} \mapsto -\sqrt{c}$ & $C_4 \mapsto \Dtilde_4$ 
			& $F_4 \mapsto G_4$ & $\gamma \mapsto -\gamma$ & $C_7 \mapsto D_7$
			& $F_7 \mapsto G_7$ & $\sqrt{c^2 - 100ab} \mapsto$ 
			& $C_4 \mapsto D_4$ & $F_4 \mapsto G_4$\\

			& $C_7 \mapsto \Dtilde_7$ & $F_7 \mapsto G_7$ && $C_9 \mapsto D_9$ & 
			$F_9 \mapsto G_9$ & $-\sqrt{c^2 - 100ab}$ & $C_6 \mapsto D_6$ 
			& $F_6 \mapsto G_6$\\

			& & $F_{11} \mapsto G_{11}$ & & & $F_{14} \mapsto G_{14}$ &
			 &  & $F_{14} \mapsto G_{14}$\\
			\hline 

			$\delta_1 \mapsto -\delta_1$ & & $F_{13} \mapsto G_{14}$
			& $\sqrt{a} \mapsto -\sqrt{a}$ & $C_5 \mapsto D_5$ 
			& $C_6 \mapsto D_6$
			& $\sqrt2 \mapsto -\sqrt2$ & $C_2 \leftrightarrow C_3$ 
			& $F_2 \leftrightarrow F_3$\\

			$\xi_i^+\leftrightarrow \xi_i^-$ & & $F_{14} \mapsto G_{13}$ &
			& $C_6\mapsto D_6$ & $F_6 \mapsto G_6$ &
			$\sqrt{-2 + 2\sqrt{2}}\leftrightarrow$ & $C_5 \mapsto \Dtilde_5$
			& $F_5 \mapsto G_5$\\

			&&&&&& $\sqrt{-2 -2\sqrt{2}}$ & $C_6 \mapsto \Dtilde_6$ 
			& $F_6 \mapsto G_6$\\
			\hline

			$\delta_2 \mapsto -\delta_2$ & & $F_{13} \leftrightarrow F_{14}$ &
			$\theta_0 \mapsto -\theta_0$ & $C_8 \leftrightarrow C_9$ 
			& $F_8 \leftrightarrow F_9$ & $\sqrt{-2 + 2\sqrt{2}} \mapsto$
			& $C_2 \mapsto \Dtilde_2$ & $F_2 \mapsto G_2$\\

			$\xi_i^+\leftrightarrow \xi_i^-$ & & & 
			$\theta_i^+\leftrightarrow\theta_i^-$ & & &
			$-\sqrt{-2 + 2\sqrt{2}}$ & $C_3\mapsto\Dtilde_3$ 
			& $F_3 \mapsto G_3$\\
			\hline

			$\xi_1^+ \mapsto -\xi_1^+$ & & $F_{13} \mapsto G_{13}$ &
			$\theta_1^+ \mapsto -\theta_1^+$ & $C_8 \mapsto D_8$ 
			& $F_8 \mapsto G_8$
			& $\sqrt{-c-10\sqrt{ab}} \mapsto$ & $C_5 \mapsto D_5$ 
			& $F_5 \mapsto G_5$\\

			&&$F_{14} \mapsto G_{14}$ & & $C_9 \mapsto D_9$ & $F_9 \mapsto G_9$ 
			& $-\sqrt{-c-10\sqrt{ab}}$ & $C_6 \mapsto D_6$ & $F_6 \mapsto G_6$\\
			\hline

			$\xi_2^+ \mapsto -\xi_2^+$ & & $F_{13} \mapsto G_{13}$ &
			$\theta_2^+ \mapsto -\theta_2^+$ & $C_8 \mapsto \Dtilde_8$ 
			& $F_8 \mapsto G_8$
			&  &  & \\

			&&$F_{14} \mapsto G_{14}$ & & $C_9 \mapsto \Dtilde_9$ 
			& $F_9 \mapsto G_9$ &&& \\
			\hline
		\end{tabular}
		\caption{The Galois action on the fibers of the genus 1 fibrations. 
		\newline The action on the splitting field is described by the action on 
		the generators listed in~\ref{eqn:gens1},~\ref{eqn:gens2},
		~\ref{eqn:gens3},~\ref{eqn:gens4},~\ref{eqn:gens5}.  If a generator of 
		$K$ is not listed, then we assume that it is fixed.  We use the same 
		convention for the curves.}\label{table:galois}
	\end{table}
	\end{landscape}


\end{document}